\documentclass{amsart}
\usepackage{latexsym}
\usepackage{amsmath}
\usepackage{amsfonts}
\usepackage{amsthm}
\usepackage{amssymb}
\usepackage[T1]{fontenc}
\usepackage{hyperref}
\usepackage{amsrefs}
\newtheorem{thm}{Theorem}
\newtheorem{cor}{Corollary}
\newtheorem{lem}{Lemma}
\newtheorem{prop}{Proposition}
\theoremstyle{remark}
\newtheorem*{rem}{Remark}
\newtheorem{ex}{Example}
\theoremstyle{definition}
\newcommand{\R}{\mathbb{R}}

\newcommand{\C}{\mathbb{C}}
\newcommand{\Ce}{\mathcal{C}}

\newcommand{\A}{\mathcal{A}}
\newcommand{\B}{\mathcal{B}}
\newcommand{\I}{\mathcal{I}}
\newcommand{\J}{\mathcal{J}}
\newcommand{\M}{\mathfrak{M}}
\newcommand{\rad}{\mathrm{rad}}

\newcommand{\f}{\mathfrak{m}}
\renewcommand{\(}{\left(} \renewcommand{\)}{\right)}

\begin{document}

\title[Delta-Sincov mappings]{Delta-Sincov mappings in Banach algebras}

\author{W\l{}odzimierz Fechner}
\address{Institute of Mathematics, Lodz University of Technology, al. Politechniki 8, 93-590 \L\'od\'z, Poland}
\email{wlodzimierz.fechner@p.lodz.pl}

\author{Aleksandra \'Swi\k{a}tczak}
\address{Institute of Mathematics, Lodz University of Technology, al. Politechniki 8, 93-590 \L\'od\'z, Poland}
\email{aleswi97@gmail.com}

\begin{abstract}
We study solutions and approximate solutions of the multiplicative Sincov equation
$$T(f, h) = T(f, g)T(g, h)$$
for mapping $T$ taking values in a commutative Banach algebra.
\end{abstract}



\maketitle 

\section{Motivation: Gr\"uss-type inequalities as an approximate multiplicative Sincov equation}

Let $[a, b]\subset \R$ be a non-degenerated interval and $f, g \in L([a,b])$ be Lebesgue integrable functions. Assume that there exist constants $m_f, m_g, M_f, M_g \in \R$ such that
\begin{equation}
\label{con}
m_f \leq f \leq M_f \quad \textrm{and} \quad m_g \leq g \leq M_g \qquad \textrm{a.e. on }[a,b].
\end{equation} 
Let us denote integral mean of a function $f$ by $I(f)$:
$$ I(f) = \frac{1}{b-a} \int\limits_a^b f(x) dx , \quad f \in L([a,b]).$$
The celebrated Gr\"uss inequality dates as early as 1935 and says that
\begin{equation}
\label{G}
\left| I(fg) -  I(f)I(g)     \right| \leq \frac14(M_f-m_f)(M_g-m_g).
\end{equation}
Nowadays several generalizations and refinements are known. For a comprehensive study of the topic and the list of references we refer the reader to a monograph by S.S. Dragomir \cite{D} and to papers by Z. Otachel \cites{Ot, Ot2, Ot3}.

A Gr\"uss-type inequality is an inequality which provides an upper bound for the expression of the form
$$\left|\left\langle f,h \right\rangle - \left\langle f,g \right\rangle \left\langle g,h \right\rangle \right|,$$
sometimes under additional assumption $\|g\|=1$, or 
$$\left|\left\langle f,h \right\rangle - \frac{1}{\|g\|^2}\left\langle f,g \right\rangle \left\langle g,h \right\rangle\right|,$$
where $f, g, h$ belong to an inner product space (see \cite{D}*{Theorems 15, 16, 17}). Cases with an inner product replaced by another functional have been studied as well, including discrete versions of the original inequality. 

In paper \cite{W1} we dealt with functionals which satisfy
\begin{equation}
|S(f,h)-S(f,g)S(g,h)|\leq c \quad \textrm{for all } f, g, h
\label{pams}
\end{equation}
with some constant $c\geq 0$. Research of \cite{W1} was motivated by Richard's inequality, which is an example of Gr\"uss-type inequality.
It says that particular solutions of \eqref{pams} are functionals of the form
$$
S(f,g) = \frac{\left\langle f,g \right\rangle}{\|f\|\cdot \|g\|}.
$$
In the main result of \cite{W1} we proved that every unbounded real- or complex-valued solution of \eqref{pams} is a solution of the Sincov equation
\begin{equation}
S(f,h)=S(f,g)S(g,h) \quad \textrm{for all } f, g, h.
\label{S}
\end{equation}
Consequently, $S$ has the representation 
\begin{equation}
S(f,g) = \frac{\varphi(f)}{\varphi(g)} \quad \textrm{for all } f, g
\label{phi}
\end{equation}
with some never-vanishing real- or complex-valued mapping $\varphi$ (see D. Gronau \cite{G}*{Theorem}). From the main result of \cite{W1} it follows that no generalization of Richard's inequality such that the functional $S$ is replaced by an unbounded one is possible.

\medskip

The multiplicative Sincov equation is important in the theory of functional equations and has several applications. Recently, G. Kiss and J. Schwaiger \cite{KS} studied its conditional version and applied their results to the problem of finding a rule of interest compounding if negative interest rates are possible. Besides this, M. Baczy\'nski et al \cites{B1, B2} investigated fuzzy implications satisfying a version of the Sincov equation. A comprehensive study of this equation with a list of references is due to D. Gronau \cite{G}.

\medskip

We will conclude this introductory section with an observation that the Sincov equation cover the equation of exponential mappings. Therefore, all our subsequent results generalize the earlier studies concerning exponential mappings.

\begin{prop}
Assume that $(X, \cdot)$ is a group, $\A$ is a commutative ring, $F\colon X\to \A$ and $S\colon X\times X\to \A$ is given by
$$S(f,g) = F\left({f}{g^{-1}}\right), \quad f, g \in X.$$
Then $F$ satisfies
$$F(f\cdot g) = F(f)F(g), \quad f, g \in X$$
if and only if $S$ satisfies \eqref{S}.
\end{prop}
Similarly, approximate solutions of the Sincov equation yield generalizations to approximatively exponential mappings. We will omit the details.

\section{Delta-Sincov mappings for complex-valued functionals}

Assume that $X$ is a nonempty set and $S$ and $F$ are a complex-valued mapping and a nonnegative mapping, respectively, acting on the product $X \times X$. In this section we will study the following functional inequality:
\begin{equation}
|S(f,h)-S(f,g)S(g,h)|\leq F(f,g)F(g,h)-F(f,h), \quad f, g, h \in X.
\label{main}
\end{equation}
Each solution of \eqref{main} will be called a delta-Sincov mapping with a control function $F$.
Clearly, if $c\geq 0$ is a fixed constant and one takes as $F$ constant map equal to $1/2(1+ \sqrt{1+4c})$, then \eqref{main} reduces to \eqref{pams}. 
Therefore, the results of this section extend the findings of \cite{W1}.

The present study is meant to fall in line with research initiated by the notion of delta-convexity by L. Vesel\'y, L. Zaj\'i\v{c}ek \cite{VZ} and then continued by several authors. Dissertation \cite{VZ} is mainly devoted to inequality
$$\left\|  F\( \frac{x+y}{2} \) - \frac{F(x)+F(y)}{2}       \right\| \leq  \frac{f(x)+f(y)}{2} - f\( \frac{x+y}{2} \) .$$
Continuous solutions $F$ of this inequality are called delta-convex. One of the theorems of \cite{VZ} says that every real-valued delta-convex function can be written as a difference of two convex functionals.
This result motivated a study of several related problems. In particular, R. Ger \cite{G} studied inequality
$$\left\|  F\( {x+y} \) - F(x)F(y)   \right\| \leq  f(x)f(y) - f\( {x+y} \),$$
using the term delta-exponential map for $F$. More recently, A. Olbry\'s dealt with delta $(s,t)$-convex mappings \cite{O1}, delta Schur-convex mappings \cite{O2}, \cite{O4} and delta-subadditive and delta-superadditive mappings \cite{O3}.

\medskip

We begin the study of \eqref{main} with an observation that inequality \eqref{main} behaves in a way similar to other inequalities motivated by the notion of delta-convexity. Proposition \ref{p1} below is an analogue to 
\cite{G}*{Proposition 1}.

\begin{prop}\label{p1}
Assume that $X$ is a nonempty set and $S\colon X \times X \to [0, + \infty) $ and $F\colon X \times X \to [0, + \infty)$ satisfy inequality \eqref{main}. If we denote $H=S+F$, then
\begin{equation}
\label{multI}
H(f,h)\leq H(f,g)H(g,h), \quad f, g, h \in X.
\end{equation}
\end{prop}
\begin{proof}
Fix $f, g, h \in X$; we have by \eqref{main}
\begin{align*}
H(f,h) &= S(f,h) + F(f,h) \leq S(f,g)S(g,h) + F(f,g)F(g,h) \\&= [H(f,g) - F(f,g) ][H(g,h) - F(g,h)] + F(f,g)F(g,h)\\&= H(f,g)H(g,h) - H(f,g)F(g,h) - F(f,g)H(g,h) + 2F(f,g)F(g,h)\\&= H(f,g)H(g,h) - S(f,g)F(g,h) - F(f,g)S(g,h)  \leq H(f,g)H(g,h).
\end{align*}
\end{proof}

We will apply an observation from \cite{W2} to exclude cases when $F$ attains zero. Indeed, by \cite{W2}*{Proposition 2}, if $F$ has a zero, then $F=0$ on $X \times X$  and, consequently, $S$ solves Sincov equation \eqref{S}. Therefore, from now on we will restrict ourselves to the case when $F$ is positive.

In our next result, we describe solutions of \eqref{main} which satisfy an additional assumption. We begin with a lemma.

\begin{lem} \label{lem}
Assume that $X$ is a nonempty set and $S\colon X \times X \to \C $ and $F\colon X \times X \to (0, + \infty)$ satisfy inequality \eqref{main}. If there exist some $f, g \in X$ such that the map 
\begin{equation}
X \ni k  \to \frac{S(g,k)}{F(f,k)}
\label{map}
\end{equation}
 is unbounded, then  for all $f, g \in X$ the map \eqref{map}  is unbounded.
\end{lem}
\begin{proof}
We will split the proof into a few steps. Let us fix $f, g \in X$ such that he map \eqref{map} is unbounded and let $h, m \in X$ be arbitrarily fixed.

\noindent\textbf{Step 1.} \emph{The map 
$$
X \ni k  \to \frac{S(g,k)}{F(h,k)}
$$
is unbounded.}

$S$ and $F$ satisfy \eqref{main}, thus 
$$|S(h,k)-S(h,f)S(f,k)|\leq F(h,f)F(f,k)-F(h,k), \quad k \in X.$$
It implies immediately the following inequality
\begin{equation}
\label{F}
F(h,k)\leq F(h,f)F(f,k), \quad k \in X.
\end{equation}
By \eqref{F} we have
$$\frac{|S(g,k)|}{F(f,k)}\leq  \frac{|S(g,k)|}{F(h,k)} F(h,f), \quad  k \in X.$$
Clearly, since the left-hand side is unbounded by assumption, then the fraction on the right-hand side is unbounded. 

\noindent\textbf{Step 2.} \emph{The map 
$$
X \ni k  \to \frac{S(h,k)}{F(h,k)}
$$
is unbounded.}

From \eqref{main} we get
$$\left| S(g,k) - S(g,h)S(h,k) \right|\leq F(g,h)F(h,k)-F(g,k)\leq F(g,h)F(h,k), \quad k \in X;$$
therefore
$$\left| \frac{S(g,k)}{F(h,k)} - S(g,h)\frac{S(h,k)}{F(h,k)} \right|\leq F(g,h)\quad k \in X.$$
By Step 1 the first fraction is unbounded, so is the second one. 

\noindent\textbf{Step 3.} \emph{The map 
$$
X \ni k  \to \frac{S(h,k)}{F(m,k)}
$$
is unbounded.}

Using \eqref{main} analogously as in Step 1 we have
$$\frac{|S(h,k)|}{F(h,k)}\leq  \frac{|S(h,k)|}{F(m,k)} F(m,h), \quad  k \in X.$$
By Step 2 the left-hand side is unbounded. Thus the fraction on the right-hand side is unbounded, which ends the proof. 
\end{proof}

The following corollary is straightforward and will be utilized at the end of the proof of the subsequent theorem.

\begin{cor}\label{c}
Under assumptions of Lemma \ref{lem}, if there exists some $g \in X$ such that $S(g,k)= 0$ for each $k \in X$, then $S=0$ on $X\times X$.
\end{cor}

Now, we are ready to prove the main result of this section.

\begin{thm} \label{thm_S}
Assume that $X$ is a nonempty set and $S\colon X \times X \to \C $ and $F\colon X \times X \to (0, + \infty)$ satisfy inequality \eqref{main}. If there exist some $f, g \in X$ such that the map \eqref{map} is unbounded, then
$S$ solves Sincov equation \eqref{S} for all $f, g, h \in X$. 
\end{thm}
\begin{proof}
Define an auxiliary functional $\Gamma\colon X \times X \to \R$ as
$$\Gamma(f, g) := F(f,g)F(g,f) - 1, \quad f,g \in X.$$ 
Using inequality \eqref{F}  we get
$$F(g, h) \leq F(g, f)F(f, h), f, g, h \in X.$$
so, 
\begin{equation}\label{gamma}
F(f,g)F(g,h) - F(f,h) \leq \Gamma(f, g)F(f,h), \quad f, g, h \in X.
\end{equation}
Now, fix arbitrary $f, g, h, k \in X$. By \eqref{main} and \eqref{gamma} we obtain
$$|S(f,k) - S(f,h)S(h,k)|\leq   F(f,h)F(h,k)-F(f,k) \leq \Gamma(f, h) F(f,k)$$
and
$$|S(f,k) - S(f,g)S(g,k)|\leq F(f,g)F(g,k)-F(f,k)\leq \Gamma(f, g) F(f,k),$$
therefore
\begin{equation}\label{1}
|  S(f,h)S(h,k)-S(f,g)S(g,k)|\leq [\Gamma(f, g) + \Gamma(f, h)] F(f,k). 
\end{equation}
Using \eqref{1}, \eqref{main} and then \eqref{gamma}  we arrive at
\begin{align*}
| S(h,k) |&\cdot |S(f,h)-S(f,g)S(g,h)| = |S(f,h)S(h,k)-S(f,g)S(g,h)S(h,k)|\\&\leq [\Gamma(f, g) + \Gamma(f, h)] F(f,k)+ | S(f,g)S(g,k)- S(f,g)S(g,h)S(h,k)|\\&\leq  [\Gamma(f, g) + \Gamma(f, h)] F(f,k)+ | S(f,g)|[F(g,h)F(h,k) - F(g,k)] \\&\leq [\Gamma(f, g) + \Gamma(f, h)] F(f,k)+ | S(f,g)| \Gamma(g, h) F(g,k).
\end{align*}
Thus, if $S(h,k)\neq 0$, then
$$
|S(f,h)-S(f,g)S(g,h)| \leq  [\Gamma(f, g) + \Gamma(f, h)]\frac{F(f,k)}{| S(h,k) |}   + |S(f,g)| \Gamma(g, h)\frac{ F(g,k)}{| S(h,k) |}  .
$$
If there were no $h \in X$ such that $S(h,k)\neq 0$ for all $k \in X$, then we would get by Corollary \ref{c} that $S=0$ on $X \times X$, contradicting the assumptions.
Finally, Lemma \ref{lem} implies that the two fractions on the right-hand side can be arbitrarily small while $f, g, h \in X$ are kept fixed. Thus we get that the left-hand side is equal to zero.
\end{proof}

\section{Delta-Sincov operators in Banach algebras}

Assume that $\A$ is a (complex) commutative Banach algebra. We will adopt the convention that Banach algebras and its subsets will be denoted by capital calligraphic letters, whereas the capital gothic fonts are reserved for sets consisting of linear multiplicative functionals over Banach algebras and small gothic letters for functionals. Thus the space of all complex homomorphisms on $\A$ (i.e. nonzero complex mappings on $\A$ which are linear and multiplicative) is denoted by $\M(\A)$. $\M(\A)$ with the weak$-^*$ topology is a compact Hausdorff space. By $\Phi\colon \A\to C_{\C}(\M(A))$ we mean the Gelfand transform, i.e.
$$\Phi(x)(\f) = \f(x), \quad \f \in \M (\A), \, x \in \A.$$
The intersection of all maximal ideals on $\A$ is the (Jacobson) radical of $\A$ and we denote it by $\rad (\A)$. It is well-known that the radical is equal to the kernel of the Gelfand transform. If $\rad (\A) = \{ 0\}$, i.e. the Gelfand transform is an isomorphism, then algebra $\A$ is termed semisimple. Further, if the algebra $\A$ is equipped with an involution $^*\colon\A\to\A$ such that 
$$\| x x^*\| = \|x\|^2, \quad x \in \A,$$
then $\A$ is called $C^*$-algebra. The celebrated Gelfand-Naimark theorem says that in the case of  $C^*$-algebras, the Gelfand transform is an isometrical isomorphism such that
$\f(x^*) = \overline{\f(x)}$ for all $\f \in \M$ and  $x \in \A$, or equivalently $\Phi(x^*) = \overline{\Phi(x)}$ for all $x \in \A$. 

\medskip

We will begin with a description of solutions of the Sincov equation in Banach algebras.

\begin{prop}\label{p2}
Assume that $X$ is a nonempty set and $T\colon X \times X\to \A$ satisfies
\begin{equation}
T(f, h) = T(f, g)T(g, h), \quad f, g, h \in X.
\label{T}
\end{equation}
Then:
\begin{itemize}
\item[(a)] If $T(f,g) = 0$ for some $f, g \in X$, then $T=0$ on $X\times X$,
\item[(b)]  If $T(f,g)$ is a zero divisor for some $f, g \in X$, then $T(f,g)$ is a zero divisor for all $f, g \in X$.
\end{itemize}
\end{prop}
\begin{proof}
Both assertions follow easily from repetitive use of \eqref{T} with the aid of the commutativity of $\A$.
\end{proof}

\begin{prop}\label{p3}
Assume that $X$ is a nonempty set. If $T\colon X \times X\to \A$ satisfies equation \eqref{T}, then
 there exist a (possibly empty) set $\M_1\subset \M(\A)$ and a mapping $\varphi\colon \M(\A)\times X \to \C\setminus\{0 \}$ such that $\Phi (T(\cdot, \cdot ))(\f) = 0$ for all $\f \in  \M(\A)\setminus \M_1$ and
\begin{equation}\label{m}
\Phi (T(f,g))(\f) = 
\f (T(f,g)) = \frac{\varphi(\f,f)}{\varphi(\f,g)}, \quad \f \in \M_1, \, f, g \in X .
\end{equation}
Conversely, if $\A$ is a semisimple Banach algebra, $\M_1\subset \M(\A)$, $\varphi\colon \M(\A)\times X \to \C\setminus\{0 \}$ and  $T\colon X \times X\to \A$ is defined by \eqref{m} and $\Phi (T(\cdot, \cdot ) )(\f) = 0$ for all $\f \in  \M(\A)\setminus \M_1$, then
$T$ satisfies \eqref{T}.
\end{prop}
\begin{proof}
To prove the first part observe that for each $\f \in \M(\A)$ the map $S:= \f \circ T$ solves \eqref{S}.
Denote $$\M_1:= \{\f \in \M(\A) : \f \circ T \neq 0  \}$$
and use Proposition \ref{p2} (a) and the representation \eqref{phi} of scalar-valued solutions recalled in the first section to derive \eqref{m}. 

To prove the converse implication observe that since $\Phi$ is an isometry, then $T$ is well-defined by \eqref{m}. Further, for each $\f \in \M_1$ and each $f, g, h \in X$ we have
\begin{align*}
\f \circ \left(T(f, h) - T(f, g)T(g, h) \right) &=\f( T(f, h)) - \f(T(f, g))\f(T(g, h))    \\ &= \frac{\varphi(\f,f)}{\varphi(\f,h)} - \frac{\varphi(\f,f)}{\varphi(\f,g)} \frac{\varphi(\f,g)}{\varphi(\f,h)} = 0, 
\end{align*}
or equivalently
$$T(f, h) - T(f, g)T(g, h) \in \ker \f.$$
Since $\f\in \M_1$ was taken arbitrary and $\f\circ T = (\Phi\circ T)(\f)=0$ for all $\f\in \M(\A)\setminus \M_1$  by assumption, then 
$$ T(f, h) - T(f, g)T(g, h)  \in \bigcap_{\f \in \M(\A)} \ker \f = \rad (\A).$$
Finally, utilizing the fact that $\A$ is a semisimple Banach algebra, we arrive at
$$T(f, h) - T(f, g)T(g, h) = 0.$$
\end{proof}

The next example show that the assumption that $\A$ is semisimple cannot be dropped. 

\begin{ex}
Let $X= (0, + \infty)$, $\A$ is a subalgebra of $2 \times 2$ complex matrices of the form $\( \begin{matrix}x & y \\ 0 & z  \end{matrix} \)$ for $x, y, z \in \C$ and  $T\colon X \times X\to \A$ is given by
$$T(f,g) := \( \begin{matrix}1 & f \\ 0 & 1  \end{matrix} \), \quad f, g \in X.$$ There are precisely two distinct linear-multiplicative functionals on $\A$, namely $$\f_1\( \begin{matrix}x & y \\ 0 & z  \end{matrix} \) = x, \quad \f_2\( \begin{matrix}x & y \\ 0 & z  \end{matrix} \) = z.$$ Therefore, 
$$ \rad (\A) = \left\{      \( \begin{matrix}0 & y \\ 0 & 0  \end{matrix}\), \quad  y \in \C  \right\},$$  the formula \eqref{m} is satisfied with $\varphi = 1$ (for both $\f_1$ and $\f_2$) and $T$ fails to satisfy equation \eqref{T}. 
\end{ex}

In the proof of the next theorem, we follow some ideas of the article R. Ger and P. \v{S}emrl \cite{GS}.

\medskip

Assume that $X$ is a nonempty set and $T\colon X \times X\to \A$ and $F\colon X \times X\to (0, +\infty)$ are arbitrary mappings. We will study the following functional inequality:
\begin{equation}
\|T(f,h)-T(f,g)T(g,h)\|\leq F(f,g)F(g,h)-F(f,h), \quad f, g, h \in X.
\label{main_T} 
\end{equation}
We do not impose any additional assumptions upon $T$. Therefore,  since we can embed $\A$ into algebra with a unit without loss of generality, we may assume that $\A$ has a unit.

\begin{thm}\label{thm_semisimple}
Let $X$ be a nonempty set and let $\A$ be a semisimple commutative Banach algebra. Assume that $T \colon X \times X \to \A$ and $F \colon X \times X \to (0, + \infty)$ satisfy inequality \eqref{main_T} together with the condition
\begin{equation}
\forall_{\f \in \M(\A)} \exists_{f, g \in X} \sup_{k \in X} \left| \frac{\f \left( T (g, k) \right)}{F(f, k)} \right| = \infty.
\label{sup}
\end{equation} 
Then $T$ solves Sincov equation \eqref{T}.
\end{thm}
\begin{proof}
Let $\f \in \M(\A)$ be fixed arbitrarily. Clearly $\|\f\| = 1$, therefore for every $f,g,h \in X$ we have
\begin{align*}
| \f \circ T(f, h) &- (\f \circ T(f, g)) ( \f \circ T(g, h)) | = \left| \f \circ \left(T(f, h) - T(f, g)T(g, h) \right) \right| \\ &\leq \|\f\| \left\| T(f,h)-T(f,g)T(g,h) \right\| \leq  F(f,g)F(g,h)-F(f,h).
\end{align*}
Define $S \colon X \times X \to \C$ as $S := \f \circ T$. Note that $S$ satisfies inequality \eqref{main}. Therefore by Theorem \ref{thm_S} we infer that $S$ satisfies equation \eqref{S}. 
Consequently, equality \eqref{m} holds true with $\varphi(\f, \cdot ) = \varphi$ given by \eqref{phi} and for every $\f\in \M(\A)$.
To finish the proof apply the second part of Proposition \ref{p3}.
\end{proof}

One can ask whether there is an analogue of the second theorem of Ger and \v{S}emrl, i.e. \cite{GS}*{Theorem 3.2}. The following example shows that a full analogue is not true. Namely, we will show that if $\A$ is a $C^*$-algebra, then the set of all $\f \in \M(\A)$ for which $(\Phi  \circ T)(\cdot)(\f)$ is bounded does not need to be closed. Consequently, it is not possible to decompose algebra $\A$ as a direct sum of two closed ideals such that, if $Q$ and $R$ are the corresponding projections, then $Q\circ T$ solves the Sincov equation and $R\circ T$ is norm bounded.

\begin{ex}
Let $$X:= \{f\in C([0,1]) : \forall_{x\in [0,1]}[x< f(x) \leq x^2+1 ]\},$$ $\A:= C_{\C}([0,1])$
and $T \colon X \times X \to \A$ and $F \colon X \times X \to (0, + \infty)$ be given by $F= (1+\sqrt{5})/2$ (a constant map) and
$$T(f,g) = \frac{f-1}{g}, \quad f, g \in X.$$
One can check that inequality \eqref{main_T} is satisfied by $T$ and $F$. Further, since each $\f \in \M(\A)$ is of the form $\f (f) = f(x)$ for some $x \in [0,1]$, then for fixed $f, g \in X$ such that $g(x) \neq 1$ for all $x \in [0,1]$ we have
$$\sup_{k \in X} \left| \frac{\f \left( T (g, k) \right)}{F(f, k)} \right| = \sup_{k \in X} \frac{|2g(x) - 2|}{(1+\sqrt{5})k(x)}.$$
We see that the supremum is finite if and only if $x \in (0,1]$. Therefore, the set of all $\f \in \M(\A)$ for which $(\Phi  \circ T)(X)(\f)$ is bounded, is not closed.
\end{ex}

In our next result, we will show that it is possible to get a relatively weak statement, which can be viewed as an analogue to \cite{GS}*{Theorem 3.2}.

\begin{prop}\label{thm_C^*}
Let $X$ be a nonempty set and let $\A$ be a $C^*$-algebra. Assume that $T \colon X \times X \to \A$ and $F \colon X \times X \to (0, + \infty)$ satisfy inequality \eqref{main_T}. 
Then there exist:  a Banach $C^*$-algebra $\B$, a normed $C^*$-algebra  $\Ce$ and  $^*$-homomorphisms  $\Lambda_1\colon\A\to\B$ and  $\Lambda_2\colon\A\to\Ce$ such that the map $\Lambda_1 \circ T$ solves Sincov equation \eqref{T},  for each $f, g \in X$ the set 
$$\left\{ \Lambda_2 \left(\frac{ T(g,k)}{F(f,k) }\right) : k \in X \right\}$$
is norm bounded in $\Ce$ and $(\Lambda_1,\Lambda_2)\colon \A \to \B \oplus \Ce$ is an isometrical $^*$-homomorphism.
\end{prop}
\begin{proof}
Define $S \colon \M(\A)\times X \times X \to \C$ as $$S(\f)(f,g) := \Phi( T(f,g))(\f) \quad \f \in \M(\A), \, f, g \in X.$$
Using the basic properties of the Gelfand transform we obtain
\begin{align*}
|S(\f)(f,h)&- S(\f)(f,g)S(\f)(g,h)| = | \Phi\big( T(f,h)-T(f,g)T(g,h)  \big)(\f) |\\&  \leq \|T(f,h)-T(f,g)T(g,h)\|\\ &\leq F(f,g)F(g,h)-F(f,h), \quad f \in \M(\A), \, f, g, h \in X.
\end{align*}
Define
$$\M_s:=\{ \f \in \M(\A) : S= S(\f) \textrm{ satisfies } \eqref{S} \}.$$
By the Gelfand-Naimark theorem, the Gelfand transform is an isometry. Hence $\M_s$ is a closed subset of $\M(\A)$ (and thus compact). 

Define $\B:=C_{\C}(\M_s)$ and $\Lambda_1\colon\A\to\B$ via
$$ C_{\C}(\M(\A))\ni \tau \to \Lambda_1(\tau) := \tau\restriction_{\M_s}\in \B.$$
Above we identified the algebra $\A$ with its image by the Gelfand transform. Directly from the definition mapping $\Lambda_1 \circ T\colon X \times X \to \B$ solves Sincov equation \eqref{T}.

Further, put $\Ce:= L^{\infty}(\M(\A)\setminus \M_s)\cap C_{\C}(\M(\A)\setminus \M_s)$ and define $\Lambda_2\colon\A\to\Ce$ as
$$ C_{\C}(\M(\A))\ni \tau \to \Lambda_2(\tau) := \tau\restriction_{\M(\A)\setminus \M_s}\in \Ce.$$
Lemma \ref{lem} implies the boundedness  in $\Ce$ of the sets $\left\{ \Lambda_2 \left(\frac{ T(g,k)}{F(f,k) }\right) : k \in X \right\}$ for all $f, g \in X$.
Finally, if the direct sum $\B\oplus\Ce$ is equipped with the norm $$\|x \oplus y\| := \max\{\|x\|, \|y\|  \}, \quad x \oplus y \in \B\oplus \Ce,$$ then it is clear that the map  $(\Lambda_1,\Lambda_2)\colon \A \to \B \oplus \Ce$ is an isometry.
\end{proof}


\section*{Declarations}

\noindent\textbf{Acknowledgements}: This article has been completed while one of the authors – Aleksandra \'Swi\k{a}tczak, was the
Doctoral Candidate in the Interdisciplinary Doctoral School at the Lodz University of Technology, Poland.

\medskip

\noindent\textbf{Ethical Approval.} Not applicable.

\medskip

\noindent\textbf{Competing interests.} The authors declare that there is no conflict of interest regarding this publication.

\medskip

\noindent\textbf{Authors' contributions.} Both authors contributed equally to the manuscript and accept its final form.

\medskip

\noindent\textbf{Funding.} The work of Aleksandra Świątczak is financed under the program "FU$^2$N - Fund for Improving Skills of Young Scientists" supporting scientific excellence of the Lodz University of Technology - grant no 5/2023. 

\medskip

\noindent\textbf{Availability of data and materials.} Any data for the methods used in the present paper are included in the paper and its references and all are available online.

\end{document}